\documentclass[10pt,A4]{article}
\usepackage[latin5]{inputenc}
\usepackage{amsmath,amssymb}
\usepackage{amsthm}
\usepackage{indentfirst}

\newtheorem{definition}{Definition}
\newtheorem{theorem}{Theorem}

\newtheorem{proposition}{Proposition}

\numberwithin{definition}{section} \numberwithin{theorem}{section}
\numberwithin{lemma}{section}\numberwithin{corollary}{section}
\numberwithin{equation}{section} \numberwithin{example}{section}
\numberwithin{proposition}{section} \numberwithin{remark}{section}
\begin{document}

\begin{center}

{\bf \Large Approximations of the Image and Integral Funnel of the \\ $L_p$ Ball under Urysohn Type Integral Operator}

\vspace{3mm}

Anar Huseyin$^1$, Nesir Huseyin$^2$, Khalik G. Guseinov$^3$

\vspace{3mm}
{\small $^1$Cumhuriyet University, Faculty of Science, Department of Statistics and Computer Sciences \\ 58140 Sivas, TURKEY

e-mail: ahuseyin@cumhuriyet.edu.tr
\vspace{2mm}

$^2$Cumhuriyet University, Faculty of Education, Department of Mathematics and Science Education \\ 58140 Sivas, TURKEY

e-mail: nhuseyin@cumhuriyet.edu.tr

\vspace{2mm}

$^3$Eskisehir Technical  University, Faculty of Science, Department of Mathematics \\ 26470 Eskisehir, TURKEY

e-mail: kguseynov@eskisehir.edu.tr}

\end{center}

\textbf{Abstract.} Approximations of the image and integral funnel of the closed ball of the space $L_p,$ $p>1,$ under Urysohn type integral operator are considered. The closed ball of the space $L_p,$ $p>1,$ is replaced by the set consisting of a finite number of piecewise-constant functions and it is proved that in the appropriate specifying of the discretization parameters, the images of defined piecewise-constant functions form an internal approximation of the image of the closed ball. Applying this result, the integral funnel of the closed ball of the space $L_p,$ $p>1,$ under Urysohn type integral operator is approximated by the set consisting of a finite number of points.

\vspace{5mm}


\vspace{5mm}


\section{Introduction}

Nonlinear integral operators arise in mathematical models of different physical, mechanical, economical, biological phenomena.  Note that the integral models have certain advantages over differential ones. For example, the outputs for such systems can be defined as continuous, even as $p$-integrable functions. In particular, the mathematical models of various input-output systems are based on the integral operator of the Urysohn type. Therefore, the construction of the set of images and integral funnel of the input functions under given integral operator is very important from the point of view of the application.

It should be noted that one of the important constructions of the theory of control systems described by ordinary differential equations is the notions of attainable set and integral funnel. Attainable set of the system is defined as the set of the points in the phase space to which the trajectories of the system at a given instant of time arrive. Integral funnel of the system is defined in the extended phase space as the set consisting of graphs of trajectories generated by all admissible control functions and is a generalization of the integral curve notion from the theory of differential equations (see, e.g., \cite{bla}, \cite{pan}). The attainable sets and integral funnel include the complete information about considered system  and often permits the construction of the trajectory with prescribed property. Different topological properties and approximate construction methods of the attainable sets and integral funnel of the given control system are the topics of a vast number of investigations. For a linear control system, the attainable sets can be described as the image of the set of control functions under appropriate Volterra, or Fredholm, or Hilbert-Schmidt integral operators.

In this paper internal approximations of the image and integral funnel of the closed ball of the space $L_p(\Omega;\mathbb{R}^m)$, $p>1,$ with radius $r$ and centered at the origin under Urysohn type integral operator are studied. The integral funnel is defined as the set of graphs of the images of all functions from $L_p$ ball. The closed ball is replaced by its subset consisting of a finite number of piecewise-constant functions. Using Steklov average of an integrable function and introducing $\Delta$-partition of a compact set, the image of the closed ball is approximated by the images of the specified finite number of piecewise-constant functions. The obtained result allows to approximate the integral funnel by the set consisting of a finite number of points.

The presented results can be applied for approximation of the set of outputs of the input-output system described by Urysohn type integral operator, where the inputs are chosen from the closed ball of the space $L_p(\Omega;\mathbb{R}^m)$, $p>1.$ Such inputs in general characterise the ones which are exhausted by consumption such as energy, fuel, finance, food, etc. (see, e.g., \cite{bel}, \cite{gusev}, \cite{kra}, \cite{sub}). Approximation of the set of outputs (or trajectories) and integral funnel of the control systems described different type integral operators and integral equations, where the input functions satisfy an integral constraint, are considered in  \cite{hus1}, \cite{hus2}, \cite{hus3}, \cite{hus4}, \cite{hus5}. In papers \cite{hus1}, \cite{hus2}, \cite{hus3}, \cite{hus4} the systems with scalar variable and continuous outputs are investigated, while in \cite{hus5}, the system with multivariable outputs is studied. Moreover, in papers \cite{hus1}, \cite{hus5}, it is assumed that the system is affine with respect to the input function, but in \cite{hus2}, \cite{hus3} \cite{hus4} it is supposed that the system is nonlinear with respect to both input and output functions. Note that in papers  \cite{hus2}, \cite{hus3}, \cite{hus4} only approximations of the sections of the set of trajectories and integral funnels are obtained. In the presented paper an approximation of the image of the closed $L_p$  ball in the space of continuous functions is given.

The paper is organized as follows. In Section 2 the basic conditions and pro\-po\-si\-tions which are used in the following arguments, are given. In Section 3, step by step way, the closed ball of the space $L_p$ is replaced by the set consisting of a finite number of piecewise-constant functions. It is proved that the set of images of defined finite number of piecewise-constant functions is an internal approximation of the image of the closed ball under considered integral operator (Theorem \ref{teo41}). An adequate approximation for integral funnel is also presented in the foregoing theorem.

\section{Preliminaries}

Consider Urysohn type integral operator
\begin{eqnarray}\label{oper}
U(x(\cdot))|(\xi)=\int_{\Omega}K(\xi,s,x(s))ds
\end{eqnarray}
where $\xi \in E,$ $s\in \Omega,$ $E \subset \mathbb{R}^{b}$, $\Omega \subset \mathbb{R}^{k}$ are compact sets, $x(\cdot)\in V_{p,r},$
\begin{eqnarray}\label{vepr}
V_{p,r}=\left\{x(\cdot)\in L_p\left(\Omega;\mathbb{R}^m\right):\left\|x(\cdot)\right\|_p \leq r\right\},
\end{eqnarray}
$p>1,$ $L_p\big(\Omega;\mathbb{R}^m\big)$ is the space of Lebesgue measurable functions $x(\cdot):\Omega\rightarrow \mathbb{R}^m$ such that
$\left\|x(\cdot)\right\|_p <+\infty,$ $\displaystyle \left\|x(\cdot)\right\|_p =\left(\int_{\Omega} \left\| x(s)\right\|^p ds\right)^{\frac{1}{p}},$
$\left\| \cdot \right\|$ denotes the Euclidean norm. It is assumed that the function $K(\cdot)$ satisfies the following conditions.

\vspace{3mm}

2.A. the function $K(\cdot):E \times \Omega \times \mathbb{R}^m \rightarrow \mathbb{R}^n$ is continuous;

\vspace{3mm}

2.B. there exists $l_0 >0$ such that
\begin{eqnarray*}
\left\|K(\xi,s,x_1)-K(\xi,s,x_2)\right\| \leq l_0 \left\|x_1-x_2\right\|
\end{eqnarray*} for every $(\xi,s,x_1)\in E \times \Omega \times \mathbb{R}^m$ and $(\xi,s,x_2)\in E \times \Omega \times \mathbb{R}^m;$

\vspace{3mm}

2.C. there exist functions $\omega(\cdot ,\cdot):\Omega \times  \mathbb{R}^m \rightarrow [0,+\infty)$, $\varphi(\cdot):[0,+\infty)  \rightarrow [0,+\infty)$ and the numbers $\beta_0 \geq 0,$ $\beta_1 \geq 0$ such that
\begin{eqnarray*}
\omega(s,x) \leq \beta_0\left\|x\right\| +\beta_1, \ \ \varphi(0)=0, \ \ \varphi(\tau) \rightarrow 0^+ \ \mbox{as} \ \tau \rightarrow 0^+
\end{eqnarray*}
for every $(s,x)\in \Omega \times \mathbb{R}^m$ and
\begin{eqnarray*}
\left\|K(\xi_1,s,x)-K(\xi_2,s,x)\right\| \leq \omega(s,x) \cdot \varphi \left(\left\|\xi_1-\xi_2\right\|\right)
\end{eqnarray*} for every $(\xi_1,s,x)\in E \times \Omega \times \mathbb{R}^m$ and $(\xi_2,s,x)\in E \times \Omega \times \mathbb{R}^m$.

\vspace{3mm}

Denote
\begin{eqnarray}\label{eq1}
\mathcal{U}_{p,r}=\left\{U(x(\cdot))|(\cdot): x(\cdot)\in V_{p,r}\right\},
\end{eqnarray}
\begin{eqnarray}\label{eq2}
\mathcal{U}_{p,r}(\xi)=\left\{y(\xi)\in \mathbb{R}^n: y(\cdot)\in \mathcal{U}_{p,r}\right\}, \ \ \xi \in E,
\end{eqnarray}
\begin{eqnarray}\label{eq3}
\mathcal{F}_{p,r}=\left\{(\xi,y(\xi))\in E \times \mathbb{R}^n: y(\cdot)\in \mathcal{U}_{p,r}\right\}.
\end{eqnarray}

It is obvious that the set $\mathcal{U}_{p,r}$ is the image of the closed ball $V_{p,r}$ under Urysohn integral operator (\ref{oper}), the set $\mathcal{F}_{p,r}$ consists of graphs of the functions from $\mathcal{U}_{p,r}.$ The set $\mathcal{F}_{p,r}$ is called integral funnel of the set $V_{p,r}$ under Urysohn integral operator (\ref{oper}).

The conditions 2.A and 2.B  imply that for each   $x(\cdot)\in V_{p,r}$ its image $U(x(\cdot))|(\cdot)$ is continuous function and the set $\mathcal{U}_{p,r}$ is a bounded subset of the space $C\left(E; \mathbb{R}^n\right)$ where $C\left(E; \mathbb{R}^n\right)$ is the space  of continuous functions $w(\cdot):E \rightarrow \mathbb{R}^n$ with norm $\left\|w(\cdot)\right\|_C=\max \left\{\left\|w(\xi)\right\|: \xi \in E\right\}.$ We set
\begin{eqnarray}\label{bc1}
B_C(1)=\left\{y(\cdot)\in C(E; \mathbb{R}^n): \left\|y(\cdot)\right\|_C\leq 1\right\},
\end{eqnarray}
\begin{eqnarray}\label{alpha*}
\alpha_*= M_0\mu(\Omega)+l_0r[\mu(\Omega)]^{\frac{p-1}{p}},
\end{eqnarray}
\begin{eqnarray}\label{betaa}
\beta_* =\beta_1 \mu(\Omega)+\beta_0 r\left[\mu(\Omega)\right]^{\frac{p-1}{p}}
\end{eqnarray} where $\mu(\Omega)$ denotes the Lebesgue measure of the set $\Omega,$ $l_0,$ $\beta_0$ and $\beta_1$ are defined in conditions 2.B and 2.C, $M_0= \max \left\{\left\|K(\xi,s,0)\right\|: (\xi,s) \in E \times \Omega \right\}.$

The Hausdorff distance between the sets $G\subset \mathbb{R}^n$, $Q\subset \mathbb{R}^n$ is denoted by symbol $H_n(G,Q)$ and the Hausdorff distance between the sets $W\subset C\left( E; \mathbb{R}^{n}\right)$, $Y\subset C\left( E; \mathbb{R}^{n}\right)$ is denoted by symbol $H_C(W,Y)$ (see, e.g., \cite{bla}, \cite{fil}). The conditions 2.A-2.C imply the validity of the following propositions.
\begin{proposition}\label{prop10} The inequality
\begin{eqnarray*}
\left\|y(\cdot)\right\|_C \leq \alpha_*
\end{eqnarray*} holds for every $y(\cdot)\in \mathcal{U}_{p,r}$ where $\alpha_*$ is defined by (\ref{alpha*}).
\end{proposition}
\begin{proposition}\label{prop11} For every $y(\cdot)\in \mathcal{U}_{p,r}$ and $\xi_1 \in E$ and $\xi_2 \in E$ the inequality
\begin{eqnarray*}
\left\|y(\xi_1) -y(\xi_2)\right\| \leq \beta_* \cdot \varphi \left( \left\| \xi_1-\xi_2\right\|\right)
\end{eqnarray*} is held, and hence
\begin{eqnarray*}
H_n \left( \mathcal{U}_{p,r}(\xi_1),\mathcal{U}_{p,r}(\xi_2)\right) \leq \beta_* \cdot \varphi \left( \left\| \xi_1-\xi_2\right\|\right)
\end{eqnarray*} is verified for every $\xi_1 \in E$ and $\xi_2 \in E$, where $\beta_*$ is defined by (\ref{betaa}).
\end{proposition}

Proposition \ref{prop11} implies that for each fixed $\xi_* \in E$ the convergence
\begin{eqnarray*}\label{beta*}
H_n \left( \mathcal{U}_{p,r}(\xi),\mathcal{U}_{p,r}(\xi_*)\right) \rightarrow 0  \ \  \mbox{as}  \ \ \xi \rightarrow \xi_*
\end{eqnarray*} is satisfied.

Now, from Arzela-Ascoli theorem, Propositions \ref{prop10} and \ref{prop11} we obtain pre\-com\-pact\-ness of the set $\mathcal{U}_{p,r}.$

\begin{proposition}\label{prop12}
The set $\mathcal{U}_{p,r}$ is a precompact subset of the space $C\left(E; \mathbb{R}^n\right).$
\end{proposition}

Now let us give definition of finite $\Delta$-partition of the set $Q\subset \mathbb{R}^{n_*}.$
\begin{definition} \label{def11} Let $Q\subset \mathbb{R}^{n_*}$ be a given set. The finite system of sets
$\Upsilon =\left\{Q_1, Q_2, \ldots, Q_T\right\}$ is said to be a finite $\Delta$-partition of given $Q$ if

\vspace{2mm}

$\mathbf{2.d_1.}$ $Q_i \subset Q$ and $Q_i$ is Lebesgue measurable for every $i=1,2,\ldots ,T$;

$\mathbf{2.d_2.}$ $Q_i\bigcap Q_j =\emptyset$ for every $i\neq j$, where $i=1,2,\ldots, T$ and $j=1,2,\ldots, T$;

$\mathbf{2.d_3.}$ $Q =\bigcup_{i=1}^{T} Q_i$;

$\mathbf{2.d_4.}$ $diam \, (Q_i) \leq \Delta$ for every $i=1,2,\ldots, T,$ where $diam \, (Q_i)=\sup\big\{\left\|x-y\right\|:$ $x\in Q_i, \ y\in Q_i \big\}$
\end{definition}

Since $\mu(Q_i)\rightarrow 0^+$ as $diam \, (Q_i)\rightarrow 0^+,$ then without loss of generality it will be assumed that for partition  $\Upsilon =\left\{Q_1, Q_2, \ldots, Q_T\right\}$  the inequality $\mu(Q_i)\leq \Delta$ is also satisfied for every $i=1,2,\ldots, T.$

\begin{proposition}\label{prop12}
Let $Q\subset \mathbb{R}^{n_*}$ be a compact set. Then for every $\Delta>0$ it has a finite $\Delta$-partition.
\end{proposition}

The next proposition will be used in the following arguments.
\begin{proposition}\label{prop13}
Let $(X,d)$ be a metric space, $P \subset X$ be a precompact set, $P_i \subset P_{i+1} \subset P$ for every $i=1,2,\ldots$ and
\begin{equation*}
H_X\left(P,\bigcup_{i=1}^{\infty}P_i\right)\leq \theta_*.
\end{equation*} Then for every $\varepsilon >0$ there exists $i(\varepsilon)>0$ such that for each $i\geq i(\varepsilon)$ the inequality
\begin{equation*}
H_X(P,P_i)\leq \theta_*+\varepsilon
\end{equation*} is satisfied where $H_X(\cdot,\cdot)$ stands for Hausdorff distance between the subsets of the metric space $(X,d).$
\end{proposition}

\vspace{3mm}

\section{Approximation}

Let $\gamma >0,$ $\Lambda =\left\{0=w_0<w_1<\ldots <w_q =\gamma\right\}$ be a uniform partition of the closed interval $\left[0,\gamma\right]$ and $\delta =w_{\lambda+1}-w_{\lambda},$ $\lambda= 0,1,\ldots, q-1$. Since $\Omega \subset \mathbb{R}^{k}$ and $E \subset \mathbb{R}^{b}$   are compact sets, then according to  Proposition \ref{prop12} for every $\Delta>0$ they have finite $\Delta$-partitions $\Upsilon_1 =\left\{\Omega_1, \Omega_2, \ldots, \Omega_M\right\}$ and $\Upsilon_2 =\left\{E_1, E_2, \ldots, E_N\right\}$ respectively.

Let $S=\left\{x\in \mathbb{R}^m: \left\|x\right\|=1\right\}$ and for given $\sigma >0$ the set $S_{\sigma}=\left\{b_1,b_2,\ldots, b_g\right\}$ be a finite $\sigma$-net on $S.$ An algorithm for specifying a finite $\sigma$-net on $S$ is given in \cite{gus2}.
We set
\begin{eqnarray}\label{eqq1}
V_{p,r}^{\gamma, \Upsilon_1, \Lambda, \sigma}&=& \big\{ x(\cdot)\in V_{p,r} : x(s)=w_{\lambda_j} b_{i_j} \ \mbox{for every} \ s \in \Omega_j, \nonumber \\
& \ & j=1,2,\ldots, M, \ w_{\lambda_j}\in \Lambda, \ b_{i_j}\in S_{\sigma} \big\},
\end{eqnarray} and let
\begin{eqnarray}\label{eqq2}
\mathcal{U}_{p,r}^{\gamma,\Upsilon_1,\Lambda, \sigma}=\left\{ U(x(\cdot))|(\cdot): x(\cdot)\in V_{p,r}^{\gamma, \Upsilon_1,\Lambda, \sigma}\right\},
\end{eqnarray}
\begin{eqnarray}\label{eqq3}
\mathcal{U}_{p,r}^{\gamma,\Upsilon_1,\Lambda, \sigma}(\xi)=\left\{y(\xi): y(\cdot)\in \mathcal{U}_{p,r}^{\gamma,\Upsilon_1,\Lambda, \sigma}\right\}, \ \xi \in E.
\end{eqnarray}

Now for each $i=1,2,\ldots ,N$ we choose an arbitrary $\xi_i \in E_i$ and denote
\begin{eqnarray}\label{eqq4}
\mathcal{F}_{p,r}^{\gamma,\Upsilon_1,\Lambda, \sigma,\Upsilon_2}=\bigcup_{i=1}^{N}\left( \xi_i, \mathcal{U}_{p,r}^{\gamma,\Upsilon_1,\Lambda, \sigma}(\xi_i)\right).
\end{eqnarray}

Note that the set $V_{p,r}^{\gamma, \Upsilon_1, \Lambda, \sigma}$ defined by (\ref{eqq1}) can be redefined as
\begin{eqnarray*}
V_{p,r}^{\gamma, \Upsilon_1, \Lambda, \sigma}&=& \big\{ x(\cdot)\in L_{p}(\Omega; \mathbb{R}^m) : x(s)=w_{\lambda_j} b_{i_j} \ \mbox{for every} \ s \in \Omega_j, \nonumber \\
& \ & j=1,2,\ldots, M, \ w_{\lambda_j}\in \Lambda, \ b_{i_j}\in S_{\sigma}, \ \sum_{j=1}^{M} \mu(\Omega_j) w_{\lambda_j}^p \leq r^p\big\}.
\end{eqnarray*}

It is obvious that the set $V_{p,r}^{\gamma, \Upsilon_1, \Lambda, \sigma}$ consists of a finite number of piecewise-constant functions, the set $\mathcal{U}_{p,r}^{\gamma, \Upsilon, \Lambda, \sigma}$ consists of a finite number of continuous functions which are the images of the functions  from the set $V_{p,r}^{\gamma, \Upsilon_1, \Lambda, \sigma}$ under operator (\ref{oper}), the set  $\mathcal{F}_{p,r}^{\gamma,\Upsilon_1,\Lambda,\sigma,\Upsilon_2}\subset \mathbb{R}^{k+n}$ is the finite union of the sets consisting of a finite number of points.

\begin{theorem}\label{teo41} For every $\varepsilon >0$ there exists $\gamma_*(\varepsilon)>0$, $\Delta_*(\varepsilon)>0$, $\delta_*(\varepsilon)>0$, $\sigma_*(\varepsilon)=\sigma_*(\varepsilon,\gamma_*(\varepsilon))>0$ such that for each $\Delta \in \left(0,\Delta_*(\varepsilon)\right],$ $\delta \in \left(0,\delta_*(\varepsilon)\right]$ and $\sigma \in (0,\sigma_*(\varepsilon)]$ the inequalities
\begin{eqnarray}\label{os1}
H_{C}\left(\mathcal{U}_{p,r}, \mathcal{U}_{p,r}^{\gamma_*(\varepsilon),\Upsilon_1, \Lambda,\sigma}\right) < \varepsilon \, ,
\end{eqnarray}
\begin{eqnarray}\label{os2}
H_{k+n}\left(\mathcal{F}_{p,r}, \mathcal{F}_{p,r}^{\gamma_*(\varepsilon),\Upsilon_1, \Lambda,\sigma, \Upsilon_2}\right) < \varepsilon
\end{eqnarray} are satisfied where the set $\mathcal{F}_{p,r}$ is defined by (\ref{eq3}), the set $\mathcal{F}_{p,r}^{\gamma_*(\varepsilon), \Upsilon_1, \Lambda,\sigma, \Upsilon_2}$ is defined by (\ref{eqq4}), $\Upsilon_1$ is a finite $\Delta$-partition of the compact set $\Omega,$ $\Upsilon_2$ is a finite $\Delta$-partition of the compact set $E,$ $\Lambda$ is a uniform partition of the closed interval  $[0,\gamma_*(\varepsilon)],$ $\delta$ is its diameter.
\end{theorem}

\begin{proof}  The proof of the theorem will be carried out in 7 steps. At first let us prove the validity of the inequality (\ref{os1}).

\textbf{Step 1.} For given $\gamma >0$ we denote
\begin{eqnarray}\label{eq4311}
V_{p,r}^{\gamma}= \big\{ x(\cdot)\in V_{p,r} : \left\|x(s) \right\| \leq \gamma \ \mbox{for every} \ s \in \Omega \ \big\},
\end{eqnarray}
\begin{eqnarray}\label{eq4411}
\mathcal{U}_{p,r}^{\gamma}=\left\{U(x(\cdot))|(\cdot): x(\cdot)\in V_{p,r}^{\gamma}\right\},
\end{eqnarray} and let
\begin{eqnarray}\label{eq7}
\kappa_*=2l_0r^p .
\end{eqnarray}

It will be proved that the inequality
\begin{eqnarray}\label{eq611}
H_C\left(\mathcal{U}_{p,r},\mathcal{U}_{p,r}^{\gamma}\right) \leq \frac{\kappa_*}{\gamma^{p-1}}.
\end{eqnarray} is held, where the sets $\mathcal{U}_{p,r}$ and $\mathcal{U}_{p,r}^{\gamma}$ are defined by (\ref{eq1}) and (\ref{eq4411}) respectively.

Let us choose an arbitrary $y(\cdot)\in \mathcal{U}_{p,r}$ which is the image of  $x(\cdot)\in V_{p,r}$ under operator (\ref{oper}) where  $V_{p,r}$ is defined by (\ref{vepr}). Define new function $x_*(\cdot):\Omega \rightarrow \mathbb{R}^m$ setting
\begin{eqnarray}\label{eq10}
x_*(s)=\left\{
\begin{array}{llll}
\displaystyle \frac{x(s)}{\left\|x(s)\right\|} \gamma \, , & \mbox{if} & \ \left\|x(s)\right\| >\gamma ,\\
x(s)  \, , & \mbox{if}  & \ \left\|x(s)\right\| \leq \gamma.
\end{array}
\right.
\end{eqnarray}
It is not difficult to show that $x_*(\cdot) \in V_{p,r}^{\gamma}$ and let the function $y_*(\cdot):E \rightarrow \mathbb{R}^n$ be the image of the function $x_*(\cdot) \in V_{p,r}^{\gamma}$ where the set $V_{p,r}^{\gamma}$ is defined by (\ref{eq4311}). It is obvious that $y_*(\cdot)\in \mathcal{U}_{p,r}^{\gamma}$. Denote $\Omega_*=\left\{ s\in \Omega: \left\|x(s)\right\| > \gamma \right\}.$
Since $x(\cdot)\in V_{p,r}$, then Tchebyshev's inequality (see, \cite{whe}, p.82) yields
\begin{eqnarray}\label{eq13}
\mu(\Omega_*)  \leq \frac{r^p}{\gamma^p} \ .
\end{eqnarray}
From (\ref{oper}), (\ref{eq7}), (\ref{eq10}), (\ref{eq13}),  Condition 2.B and H\"{o}lder's inequality we obtain that
\begin{eqnarray*}
\displaystyle \left\|y(\xi) -y_*(\xi) \right\| \leq  \int_{\Omega_*}  l_0 \left\|x(s)-x_*(s)\right\| ds \leq 2rl_0[\mu(\Omega_*)]^{\frac{p-1}{p}} \leq \frac{2l_0 r^p}{\gamma^{p-1}}= \frac{\kappa_*}{\gamma^{p-1}}
\end{eqnarray*} for every $\xi \in E$ and consequently
\begin{eqnarray} \label{eq14} \displaystyle
\left\| y(\cdot) -y_*(\cdot)\right\|_{C} \leq \frac{\kappa_*}{\gamma^{p-1}} \ .
\end{eqnarray}

Since $y(\cdot)\in \mathcal{U}_{p,r}$ is arbitrarily chosen, then (\ref{eq14}) implies  that
\begin{eqnarray}\label{eq15}
\displaystyle \mathcal{U}_{p,r} \subset \mathcal{U}_{p,r}^{\gamma} +\frac{\kappa_*}{\gamma^{p-1}}B_C(1)
\end{eqnarray} where $B_C(1)$ is defined by (\ref{bc1}).
The inclusion $\mathcal{U}_{p,r}^{\gamma} \subset \mathcal{U}_{p,r}$ and (\ref{eq15}) yield the proof of the inequality (\ref{eq611}).

Let
\begin{eqnarray}\label{eq60}
\gamma_*(\varepsilon) =\left(\frac{10k_*}{\varepsilon}\right)^{\frac{1}{p-1}}
\end{eqnarray} where $k_*$ is defined by (\ref{eq7}). (\ref{eq611}) and (\ref{eq60}) imply that
\begin{eqnarray}\label{eq61}
H_C\left(\mathcal{U}_{p,r},\mathcal{U}_{p,r}^{\gamma_*(\varepsilon)}\right) \leq \frac{\varepsilon}{10} \, .
\end{eqnarray}

\vspace{3mm}

\textbf{Step 2.} For given $\gamma_*(\varepsilon) >0$ we denote
\begin{eqnarray}\label{eq4312}
V_{p,r}^{\gamma_*(\varepsilon), Lip}= \big\{ x(\cdot)\in V_{p,r}^{\gamma_*(\varepsilon)} : x(\cdot) : \Omega \rightarrow \mathbb{R}^m \ \mbox{is Lipschitz continuous} \big\},
\end{eqnarray} and let
\begin{eqnarray}\label{eq4412}
\mathcal{U}_{p,r}^{\gamma_*(\varepsilon),Lip}=\left\{U(x(\cdot))|(\cdot): x(\cdot)\in V_{p,r}^{\gamma_*(\varepsilon),Lip}\right\}.
\end{eqnarray}

In this step it will be shown that the equality
\begin{eqnarray}\label{eq2011}
h_C\left( \mathcal{U}_{p,r}^{\gamma_*(\varepsilon)}, \mathcal{U}_{p,r}^{\gamma_*(\varepsilon),Lip}\right) =0
\end{eqnarray}
is verified where the sets $\mathcal{U}_{p,r}^{\gamma_*(\varepsilon)}$ and $\mathcal{U}_{p,r}^{\gamma_*(\varepsilon),Lip}$ are defined by (\ref{eq4411}) and (\ref{eq4412}) respectively.

For $w\in \mathbb{R}^k$ and $\alpha >0$  we denote $B_k(w,\alpha)=\left\{y\in \mathbb{R}^k: \left\|y-w\right\| <  \alpha \right\},$  $\overline{B}_k(w,\alpha)=\big\{y\in \mathbb{R}^k: \left\|y-w\right\| \leq  \alpha \big\}.$

Let us choose an arbitrary  $x(\cdot)\in V_{p,r}^{\gamma_*(\varepsilon)}$ and let  $h \in (0,1)$ be fixed. Now we define the function $x_{h}(\cdot):\Omega \rightarrow \mathbb{R}^m$ setting
\begin{eqnarray}\label{eq23}
x_{h}(s)= \frac{1}{v_{h}} \int_{B_k(s,h)}x(\nu)d\nu \, , \ s\in \Omega
\end{eqnarray} where $v_h$ is the Lebesgue measure of the ball centered at the origin with radius $h$ in the space $\mathbb{R}^k$, i.e. $v_h =\mu (B_k(0,h)).$  Note that if $\nu \not \in \Omega $, then in the equality (\ref{eq23}) we assume that $x(\nu)=0.$ It is known that
\begin{eqnarray}\label{eq24}
v_h=c_* \cdot h^k,
\end{eqnarray}
where $\displaystyle c_*=\frac{\pi^{\displaystyle k/2}}{\Gamma \left(\displaystyle \frac{k}{2}+1 \right)} \, ,$ \ $\Gamma(\cdot)$ is Euler's function.

The function $x_h(\cdot)$ is called the Steklov average of function $x(\cdot).$ Note that in \cite{kan} (Lemma 1, p.317) it is proved that if $x(\cdot)\in L_p\left(\Omega; \mathbb{R}^m\right),$ then $x_h(\cdot)\in C\left(\Omega; \mathbb{R}^m\right).$ Here we follow the proof scheme from \cite{kan} to prove that for $x(\cdot)\in V_{p,r}^{\gamma_*(\varepsilon)}$ the inclusion $x_h(\cdot)\in V_{p,r}^{\gamma_*(\varepsilon), Lip}$ is satisfied where $V_{p,r}^{\gamma_*(\varepsilon), Lip}$ is defined by (\ref{eq4312}).

Since $x(\cdot)\in V_{p,r}^{\gamma_*(\varepsilon)}$, then from (\ref{eq23}) we have that $\left\|x_h(s)\right\|\leq \gamma_*(\varepsilon)$
for every $s \in \Omega.$ Applying H\"{o}lder's inequality and taking into consideration that $\left\|x(\cdot)\right\|_p\leq r$, it is not difficult to verify that the inequality $\left\|x_h(\cdot)\right\|_p \leq r$ holds. Thus, we have that $x_h(\cdot)\in V_{p,r}^{\gamma_*(\varepsilon)}.$ Now, let us prove that the function $x_h(\cdot):\Omega \rightarrow \mathbb{R}^k$ is Lipschitz continuous. Choose arbitrary $s_1 \in \Omega$ and $s_2 \in \Omega$.

For fixed $h\in (0,1)$ and chosen $s_1 \in \Omega$ and $s_2 \in \Omega$ two cases are possible:

Case 1: Let $\left\|s_2-s_1\right\| \geq 2h.$ From inclusion $x(\cdot)\in V_{p,r}^{\gamma_*(\varepsilon)}$,  H\"{o}lder's inequality, (\ref{eq23}) and (\ref{eq24}), it follows that
\begin{eqnarray}\label{eq28}
&& \left\|x_h(s_2)-x_h(s_1)\right\| \leq   \frac{r}{\left[c_* h^{k+p}\right]^{\frac{1}{p}}} \cdot \left\|s_2-s_1\right\|
\end{eqnarray} where $c_*$ is defined in (\ref{eq24}).

Case 2. Let $\left\|s_2-s_1\right\| < 2h.$  It is possible to show that  (\ref{eq23}), (\ref{eq24}) and inclusion $x(\cdot)\in V_{p,r}^{\gamma_*(\varepsilon)}$ imply the validity of the inequality
\begin{eqnarray}\label{eq34}
\left\|x_h(s_2)-x_h(s_1)\right\|\leq  \frac{k\gamma_*(\varepsilon)}{h} \left\|s_2-s_1\right\|.
\end{eqnarray}

Denote
\begin{eqnarray}\label{eq35}
\chi (\gamma_*(\varepsilon),h)=\max \left\{\frac{r}{\left[c_* h^{k+p}\right]^{\frac{1}{p}}} \, , \ \frac{k\gamma_*(\varepsilon)}{h} \right\}.
\end{eqnarray}

By virtue of (\ref{eq28}), (\ref{eq34}) and (\ref{eq35}) we have that for every $\xi_1\in \Omega,$ $\xi_2 \in \Omega$ the inequality
\begin{eqnarray}\label{eq36}
\left\|x_h(s_2)-x_h(s_1)\right\|\leq \chi (\gamma_*(\varepsilon), h) \cdot \left\|s_2-s_1\right\|
\end{eqnarray} is satisfied. This means that for fixed $\varepsilon>0$ and $h\in (0,1)$ the function $x_h(\cdot):\Omega \rightarrow \mathbb{R}^m$ is Lipschitz continuous. Since $x_h(\cdot) \in V_{p,r}^{\gamma_*(\varepsilon)},$ then from (\ref{eq4312}) and (\ref{eq36}) we conclude that $x_h(\cdot) \in V_{p,r}^{\gamma_*(\varepsilon), Lip}.$

Let $\rho >0$ be an arbitrarily chosen number. We will prove that
\begin{eqnarray}\label{eq36*}
h_C\left( \mathcal{U}_{p,r}^{\gamma_*(\varepsilon)}, \mathcal{U}_{p,r}^{\gamma_*(\varepsilon),Lip}\right) \leq \rho
\end{eqnarray} where  $\mathcal{U}_{p,r}^{\gamma_*(\varepsilon), Lip}$ is defined by (\ref{eq4412}).

Choose an arbitrary $\tilde{y}(\cdot)\in \mathcal{U}_{p,r}^{\gamma_*(\varepsilon)}$ which is the image of $\tilde{x}(\cdot)\in V_{p,r}^{\gamma_*(\varepsilon)}$ under operator (\ref{oper}).  Choose a sequence $\left\{h_j\right\}_{j=1}^{\infty}$ such that $h_j \in (0,1)$ for every $j=1,2,\ldots$ and $h_j \rightarrow 0^+$ as $j\rightarrow \infty$. Define new function $x_j(\cdot):\Omega \rightarrow \mathbb{R}^m$ setting
\begin{eqnarray*}
x_{j}(s)= \frac{1}{v_{h_j}} \int_{B_k(s,h_j)}\tilde{x}(\nu)d\nu, \ \ s \in \Omega,
\end{eqnarray*} which is Steklov average of the function $\tilde{x}(\cdot)$ for $h_j\in (0,1).$ Then we have that  $x_j(\cdot)\in V_{p,r}^{\gamma_*(\varepsilon),Lip}$ for every $j=1,2,\ldots$ According to the Lemma 4 from \cite{kan} (p.319) we have that $\left\|x_j(\cdot)-\tilde{x}(\cdot)\right\|_p \rightarrow 0$ as $j\rightarrow +\infty$, and hence for $\rho >0$ there exists $j_*>0$ such that
\begin{eqnarray}\label{eq39}
\left\|x_{j_*}(\cdot)-\tilde{x}(\cdot)\right\|_p \leq \frac{\rho}{l_0 \left[\mu(\Omega)\right]^{\frac{p-1}{p}}}
\end{eqnarray} where $l_{0}$ is defined in condition 2.B.

Let $\tilde{y}_*(\cdot):E\rightarrow \mathbb{R}^n$ be the image of $x_{j_*}(\cdot)\in V_{p,r}^{\gamma_*(\varepsilon),Lip}$ under operator (\ref{oper}). Then $\tilde{y}_*(\cdot)\in \mathcal{U}_{p,r}^{\gamma_*(\varepsilon), Lip}$ and (\ref{oper}), (\ref{eq39}), condition 2.B and H\"{o}lder's inequality imply that
\begin{eqnarray*}
\left\|\tilde{y}(\xi)-\tilde{y}_*(\xi)\right\| \leq l_0\int_{\Omega} \left\|\tilde{x}(s)-x_{j_*}(s)\right\|ds  \leq l_0 \left[\mu(\Omega)\right]^{\frac{p-1}{p}}\left\|\tilde{x}(\cdot)-x_{j_*}(\cdot)\right\|_p \leq \rho
\end{eqnarray*}for every $\xi \in E$, and hence $\left\|\tilde{y}(\cdot)-\tilde{y}_*(\cdot)\right\|_C  \leq \rho.$ This means that
\begin{eqnarray}\label{eq41}
\mathcal{U}_{p,r}^{\gamma_*(\varepsilon)} \subset \mathcal{U}_{p,r}^{\gamma_*(\varepsilon), Lip}+\rho B_C(1).
\end{eqnarray}
Since $\mathcal{U}_{p,r}^{\gamma_*(\varepsilon),Lip} \subset \mathcal{U}_{p,r}^{\gamma_*(\varepsilon)}$, then from (\ref{eq41}) it follows the validity of the inequality (\ref{eq36*}). Finally, since $\rho >0$ is arbitrarily chosen, then (\ref{eq36*}) yields the validity of the equality (\ref{eq2011}).

\vspace{3mm}

\textbf{Step 3.} For given $\gamma_*(\varepsilon) >0$ and integer $R>0$ we denote
\begin{eqnarray}\label{eq43}
V_{p,r}^{\gamma_*(\varepsilon), Lip, R}= \big\{ x(\cdot)\in V_{p,r}^{\gamma_*(\varepsilon), Lip} : \ \mbox{Lipschitz constant of} \ x(\cdot) \nonumber \\ \mbox{is not greater than} \ R \big\},
\end{eqnarray}
\begin{eqnarray}\label{eq44}
\mathcal{U}_{p,r}^{\gamma_*(\varepsilon),Lip,R}=\left\{U(x(\cdot))|(\cdot): x(\cdot)\in V_{p,r}^{\gamma_*(\varepsilon),Lip,R}\right\}.
\end{eqnarray}

It is not difficult to verify that $V_{p,r}^{\gamma_*(\varepsilon), Lip, R}\subset C(\Omega; \mathbb{R}^m)$ and $\mathcal{U}_{p,r}^{\gamma_*(\varepsilon), Lip, R}\subset C(E; \mathbb{R}^n)$ are compact sets. Moreover, one can show that $V_{p,r}^{\gamma_*(\varepsilon), Lip}= \displaystyle \bigcup_{R=1}^{+\infty}V_{p,r}^{\gamma_*(\varepsilon), Lip, R}$ and hence
\begin{eqnarray}\label{eq411}
\mathcal{U}_{p,r}^{\gamma_*(\varepsilon),Lip} = \bigcup_{R=1}^{+\infty} \mathcal{U}_{p,r}^{\gamma_*(\varepsilon), Lip,R}
\end{eqnarray} where $V_{p,r}^{\gamma_*(\varepsilon), Lip, R}$ and $\mathcal{U}_{p,r}^{\gamma_*(\varepsilon), Lip, R}$ are defined by (\ref{eq43}) and (\ref{eq44}) respectively.

Now, from (\ref{eq61}), (\ref{eq2011}) and (\ref{eq411}) it follows that

\begin{eqnarray}\label{equ61}
H_C\left(\mathcal{U}_{p,r},\bigcup_{R=1}^{\infty}\mathcal{U}_{p,r}^{\gamma_*(\varepsilon), Lip, R}\right) \leq \frac{\varepsilon}{10} \, .
\end{eqnarray}

According to the Proposition \ref{prop10} we have that $\mathcal{U}_{p,r}\subset C(E; \mathbb{R}^n)$ is a precompact set and $\mathcal{U}_{p,r}^{\gamma_*(\varepsilon), Lip, R}\subset \mathcal{U}_{p,r}^{\gamma_*(\varepsilon), Lip, R+1}\subset \mathcal{U}_{p,r}$ for every $R=1,2,\ldots$  Then by virtue of (\ref{equ61}) and  the Proposition \ref{prop12} we have that for $\frac{\varepsilon}{10}$ there exists integer $R_*(\varepsilon)>0$ such that
\begin{eqnarray*}
H_C\left(\mathcal{U}_{p,r},\mathcal{U}_{p,r}^{\gamma_*(\varepsilon), Lip, R}\right) \leq \frac{\varepsilon}{10}+\frac{\varepsilon}{10}=\frac{\varepsilon}{5}
\end{eqnarray*} for every $R\geq R_*(\varepsilon),$ and consequently
\begin{eqnarray}\label{equ62}
H_C\left(\mathcal{U}_{p,r},\mathcal{U}_{p,r}^{\gamma_*(\varepsilon), Lip, R_*(\varepsilon)}\right) \leq \frac{\varepsilon}{5}.
\end{eqnarray}

\textbf{Step 4.}  For given $\Delta$-partition $\Upsilon_1 =\left\{\Omega_1, \Omega_2, \ldots, \Omega_M\right\}$ of the compact set $\Omega$ we set
\begin{eqnarray}\label{eqq700}
V_{p,r}^{\gamma_*(\varepsilon), \Upsilon_1}= \big\{ x(\cdot)\in V_{p,r}^{\gamma_*(\varepsilon)} : x(s)=x_j \ \mbox{for every} \ s \in \Omega_j, \ j=1,2,\ldots, M \big\},
\end{eqnarray} and let
\begin{eqnarray*}
\mathcal{U}_{p,r}^{\gamma_*(\varepsilon),\Upsilon_1}=\left\{U(x(\cdot))|(\cdot): x(\cdot)\in V_{p,r}^{\gamma_*(\varepsilon),\Upsilon_1}\right\},
\end{eqnarray*}
where $\gamma_*(\varepsilon)$ is defined by (\ref{eq60}).

Now let us choose an arbitrary $\hat{y}(\cdot)\in \mathcal{U}_{p,r}^{\gamma_*(\varepsilon),Lip, R_*(\varepsilon)}$ which is the image of $\hat{x}(\cdot)\in V_{p,r}^{\gamma_*(\varepsilon),Lip, R_*(\varepsilon)}$ where $V_{p,r}^{\gamma_*(\varepsilon),Lip, R_*(\varepsilon)}$ is defined by (\ref{eq43}). This means that
\begin{eqnarray}\label{eqq8}
\begin{array}{llll}
\left\|\hat{x}(\cdot)\right\|_p \leq r,  \ \ \left\|\hat{x}(s)\right\| \leq \gamma_*(\varepsilon) \ \mbox{for every} \ s\in \Omega, \\
\left\|\hat{x}(s_*)-\hat{x}(s^*)\right\| \leq R_*(\varepsilon)\left\|s_*-s^*\right\| \ \mbox{for every} \ s_*\in \Omega, \ s^* \in \Omega.
\end{array}
\end{eqnarray}

Define the function $\hat{x}_*(\cdot):\Omega \rightarrow \mathbb{R}^m$ setting
\begin{eqnarray}\label{eqq9}
\hat{x}_*(s)=\frac{1}{\mu(\Omega_j)}\int_{\Omega_j}\hat{x}(\nu)d\nu,  \ s \in \Omega_j, \ j=1,2,\ldots, M.
\end{eqnarray}

From (\ref{eqq8}) it follows that
\begin{eqnarray}\label{eqq10}
\left\|\hat{x}_*(s)\right\|\leq \gamma_*(\varepsilon)
\end{eqnarray} for every $s\in \Omega.$

From (\ref{eqq9}) and H\"{o}lder's inequality we obtain
\begin{eqnarray*}
\left\|\hat{x}_*(s)\right\| \leq \frac{1}{\left[\mu(\Omega_j)\right]^{\frac{1}{p}}} \left( \int_{\Omega_j}\left\|\hat{x}(\nu)\right\|^p d\nu \right)^{\frac{1}{p}}
\end{eqnarray*} for every $s\in \Omega_j,$ and hence
\begin{eqnarray*}
 \int_{\Omega_j}\left\|\hat{x}_*(\nu)\right\|^p d\nu \leq \int_{\Omega_j}\left\|\hat{x}(\nu)\right\|^p d\nu
\end{eqnarray*} for every $j=1,2,\ldots, M.$ The last inequality and (\ref{eqq8}) imply that
\begin{eqnarray}\label{eqq11}
\int_{\Omega}\left\|\hat{x}_*(\nu)\right\|^p d\nu &=&\sum_{j=1}^{M} \int_{\Omega_j}\left\|\hat{x}_*(\nu)\right\|^pd\nu \leq \sum_{j=1}^{M} \int_{\Omega_j}\left\|\hat{x}(\nu)\right\|^pd\nu \nonumber \\ &=&\int_{\Omega}\left\|\hat{x}(\nu)\right\|^pd\nu   \leq r^p .
\end{eqnarray}

(\ref{eqq9}), (\ref{eqq10}) and (\ref{eqq11}) yield $\hat{x}_*(\cdot)\in V_{p,r}^{\gamma_*(\varepsilon),\Upsilon_1}.$

Let us choose an arbitrary $s\in \Omega$ and fix it. Since $\Upsilon_1=\left\{\Omega_1,\Omega_2,\ldots,\Omega_M\right\}$ is a finite $\Delta$-partition of $\Omega,$ then on behalf of Definition \ref{def11} we have that there exists $j_*=1,2,\ldots, M$ such that $s\in \Omega_{j_*}$ where $diam (\Omega_{j_*})\leq \Delta.$ Inclusion $\hat{x}(\cdot)\in V_{p,r}^{\gamma_*(\varepsilon),Lip, R_*(\varepsilon)}$ and (\ref{eqq9}) imply
\begin{eqnarray}\label{eqq12}
\left\|\hat{x}(s)-\hat{x}_*(s)\right\| &\leq & \frac{1}{\mu(\Omega_{j_*})}\int_{\Omega_{j_*}}\left\|\hat{x}(s)-\hat{x}(\nu)\right\| d\nu \nonumber \\ &\leq & \frac{1}{\mu(\Omega_{j_*})}R_*(\varepsilon)\int_{\Omega_{j_*}}\left\|s-\nu \right\| d\nu  \leq R_*(\varepsilon)\Delta \, .
\end{eqnarray}

Now let $\hat{y}_*(\cdot)$ be the image of $\hat{x}_*(\cdot)$ under operator (\ref{oper}). Then $\hat{y}_*(\cdot)\in \mathcal{U}_{p,r}^{\gamma_*(\varepsilon),\Upsilon_1}$ and from condition 2.B and (\ref{eqq12}) we have
\begin{eqnarray*}
\left\|\hat{y}(\xi)-\hat{y}_*(\xi)\right\| \leq \int_{\Omega}l_0 \left\|\hat{x}(\nu)-\hat{x}_*(\nu)\right\|d\nu \leq l_0\mu(\Omega)R_*(\varepsilon) \Delta
\end{eqnarray*} for every $\xi \in E$ and consequently
\begin{eqnarray*}
\left\|\hat{y}(\cdot)-\hat{y}_*(\cdot)\right\|_C \leq  l_0\mu(\Omega)R_*(\varepsilon) \Delta.
\end{eqnarray*}

Since $\hat{y}(\cdot)\in \mathcal{U}_{p,r}^{\gamma_*(\varepsilon),Lip, R_*(\varepsilon)}$ is arbitrarily chosen, $\hat{y}_*(\cdot)\in \mathcal{U}_{p,r}^{\gamma_*(\varepsilon),\Upsilon_1}$, the last inequality implies that
\begin{eqnarray}\label{eqq13}
\mathcal{U}_{p,r}^{\gamma_*(\varepsilon),Lip, R_*(\varepsilon)} \subset
\mathcal{U}_{p,r}^{\gamma_*(\varepsilon),\Upsilon_1} + l_0\mu(\Omega)R_*(\varepsilon) \Delta B_C(1)
\end{eqnarray} where $B_C(1)$ is defined by (\ref{bc1}).

Since $\varphi(\tau) \rightarrow 0^+$ as $\tau \rightarrow 0^+$, then there exists $\Delta_1(\varepsilon)>0$  such that
\begin{eqnarray}\label{eqq15*}
\varphi(\Delta)\leq \frac{\varepsilon}{10\beta_*}
\end{eqnarray} for every $\Delta \in (0,\Delta_1(\varepsilon)],$ where $\varphi(\cdot)$ is given in condition 2.C, $\beta_*$ is defined by (\ref{betaa}). Denote
\begin{eqnarray}\label{eqq15}
\Delta_*(\varepsilon) = \min \left\{ \frac{\varepsilon}{10 l_0 \mu(\Omega)R_*(\varepsilon)}, \frac{\varepsilon}{10}, \Delta_1(\varepsilon) \right\}.
\end{eqnarray}

From (\ref{equ62}), (\ref{eqq13}) and (\ref{eqq15}) it follows that for every $\Delta \in (0, \Delta_*(\varepsilon)]$ the inclusion
\begin{eqnarray*}
\mathcal{U}_{p,r} \subset \mathcal{U}_{p,r}^{\gamma_*(\varepsilon),\Upsilon_1} + \frac{3\varepsilon}{10}B_C(1)
\end{eqnarray*} is satisfied.

Since $\mathcal{U}_{p,r}^{\gamma_*(\varepsilon),\Upsilon_1} \subset \mathcal{U}_{p,r},$ then the last inclusion yield that for every $\Delta \in (0, \Delta_*(\varepsilon))$ the inequality
\begin{eqnarray}\label{eqq16}
H_C \left(\mathcal{U}_{p,r}, \mathcal{U}_{p,r}^{\gamma_*(\varepsilon),\Upsilon_1}\right) \leq \frac{3\varepsilon}{10}
\end{eqnarray} is held.

\vspace{3mm}

\textbf{Step 5.} For given $\gamma_*(\varepsilon)>0,$ $\Delta$-partition $\Upsilon_1 =\left\{\Omega_1, \Omega_2, \ldots , \Omega_M\right\}$ of the set $\Omega$ and uniform $\delta$-partition $\Lambda =\left\{0=w_0<w_1<\ldots <w_q =\gamma_*(\varepsilon)\right\}$ of the closed interval $\left[0,\gamma_*(\varepsilon)\right]$  we set
\begin{eqnarray}\label{v1}
V_{p,r}^{\gamma_*(\varepsilon), \Upsilon_1, \Lambda}&=& \big\{ x(\cdot)\in V_{p,r}^{\gamma_*(\varepsilon),\Upsilon_1} : x(s)=x_j \ \mbox{for every} \ s \in \Omega_j, \nonumber \\
&& \qquad \left\|x_j\right\| \in \Lambda, \ j=1,2,\ldots, M \big\},
\end{eqnarray} where $\delta =w_{\lambda+1}-w_{\lambda},$ $\lambda= 0,1,\ldots, q-1$, $\gamma_*(\varepsilon)$ is defined by (\ref{eq60}), and let
\begin{eqnarray*}
\mathcal{U}_{p,r}^{\gamma_*(\varepsilon),\Upsilon_1,\Lambda}=\left\{U(x(\cdot))|(\cdot): x(\cdot)\in V_{p,r}^{\gamma_*(\varepsilon),\Upsilon_1, \Lambda}\right\}.
\end{eqnarray*}

Let us choose an arbitrary $y_0(\cdot)\in \mathcal{U}_{p,r}^{\gamma_*(\varepsilon),\Upsilon_1}$ which is the image of $x_0(\cdot) \in V_{p,r}^{\gamma_*(\varepsilon),\Upsilon_1}.$ Then by virtue of (\ref{eqq700}) we have that
\begin{eqnarray}\label{eqq18}
\begin{array}{lllll}
\displaystyle x_0(s)=x_j, \ \left\|x_j \right\| \leq \gamma_*(\varepsilon) \ \mbox{for every} \ s\in \Omega_j \ \mbox{and} \ j=1,2, \ldots, M,  \\ \displaystyle \sum_{j=1}^{M} \mu(\Omega_j) \left\|x_j\right\|^p \leq r^p .
\end{array}
\end{eqnarray}

The inequality $0\leq \left\|x_j \right\| \leq \gamma_*(\varepsilon)$ implies that if $\left\|x_j \right\| <\gamma_*(\varepsilon)$, then there exists $w_{\lambda_j}\in \Lambda$ such that
\begin{eqnarray}\label{eqq19}
\left\|x_j \right\| \in \left[w_{\lambda_j}, w_{\lambda_j +1}\right).
\end{eqnarray}
Define the function $\tilde{x}_0(\cdot):\Omega \rightarrow \mathbb{R}^m$ setting
\begin{eqnarray}\label{eqq20}
\tilde{x}_0(s)=\left\{
\begin{array}{llll}
\displaystyle \frac{x_j}{\left\|x_j\right\|}w_{\lambda_j} \, , &  \mbox{if}  & 0< \left\|x_j \right\| < \gamma_*(\varepsilon),\\
x_j \, , & \mbox{if} & \left\|x_j \right\| =0 \  \mbox{or} \ \left\|x_j\right\| = \gamma_*(\varepsilon)
\end{array}
\right.
\end{eqnarray} where $s\in \Omega_j,$ $w_{\lambda_j}\in \Lambda$ is defined in (\ref{eqq19}), $j=1,2,\ldots,M.$ From (\ref{eqq19}) and (\ref{eqq20})  it follows that if $0< \left\|x_j \right\| < \gamma_*(\varepsilon),$ then
\begin{eqnarray*}
\left\| x_0(s)-\tilde{x}_0(s)\right\|=\left\| x_j -\frac{x_j}{\left\|x_j\right\|}w_{\lambda_j}\right\| = \left\|x_j\right\| -w_{\lambda_j} \leq \delta
\end{eqnarray*} for every $s\in \Omega_j$ where $\delta=w_{\lambda +1}-w_{\lambda},$ $\lambda =0,1,\ldots, q-1,$ is diameter of the uniform partition $\Lambda .$ If $\left\|x_j \right\| =0$ or $\left\|x_j\right\| = \gamma_*(\varepsilon)$ then $\left\| x_0(s)-\tilde{x}_0(s)\right\|=0$ for every $s\in \Omega_j.$ Thus we have that
\begin{eqnarray}\label{eqq21}
\left\| x_0(s)-\tilde{x}_0(s)\right\| \leq \delta
\end{eqnarray} for every $s\in \Omega.$

(\ref{eqq18}), (\ref{eqq19}) and (\ref{eqq20}) yield that $\left\|\tilde{x}_0(s)\right\| \leq \left\|x_0(s)\right\|$ for every $s\in \Omega$ and consequently $\left\|\tilde{x}_0(s)\right\| \leq \gamma_*(\varepsilon)$ for every $s\in \Omega$ and $\left\|\tilde{x}_0(\cdot)\right\|_p \leq \left\|x_0(\cdot)\right\|_p  \leq r^p.$ Thus we obtain that $\tilde{x}_0(\cdot)\in V_{p,r}^{\gamma_*(\varepsilon),\Upsilon_1, \Lambda}$ and let $\tilde{y}_0(\cdot):E\rightarrow \mathbb{R}^n$ be the image of the function $\tilde{x}_0(\cdot)$ under operator (\ref{oper}). Then $\tilde{y}_0(\cdot)\in \mathcal{U}_{p,r}^{\gamma_*(\varepsilon),\Upsilon_1, \Lambda}$ and
the condition 2.B, (\ref{oper}) and (\ref{eqq21}) imply that
\begin{eqnarray*}
\left\| y_0(\xi)-\tilde{y}_0(\xi)\right\| \leq l_0 \int_{\Omega}\left\|x_0(s)-\tilde{x}_0(s)\right\| ds \leq l_0 \mu(\Omega)\delta
\end{eqnarray*} for every $\xi \in E,$ and hence
\begin{eqnarray}\label{eqq22}
\left\| y_0(\cdot)-\tilde{y}_0(\cdot)\right\|_C \leq  l_0 \mu(\Omega)\delta.
\end{eqnarray} Since $y_0(\cdot)\in \mathcal{U}_{p,r}^{\gamma_*(\varepsilon),\Upsilon_1}$ is an arbitrarily chosen function and $\tilde{y}_0(\cdot)\in \mathcal{U}_{p,r}^{\gamma_*(\varepsilon),\Upsilon_1, \Lambda},$ then (\ref{eqq22}) yields that
\begin{eqnarray*}
\mathcal{U}_{p,r}^{\gamma_*(\varepsilon),\Upsilon_1, \Lambda} \subset \mathcal{U}_{p,r}^{\gamma_*(\varepsilon),\Upsilon_1} + l_0 \mu(\Omega)\delta \cdot B_C(1).
\end{eqnarray*} Taking into consideration, that $\mathcal{U}_{p,r}^{\gamma_*(\varepsilon),\Upsilon_1, \Lambda} \subset  \mathcal{U}_{p,r}^{\gamma_*(\varepsilon),\Upsilon_1},$ from the last inclusion we obtain that
\begin{eqnarray}\label{eqq23}
H_C\left( \mathcal{U}_{p,r}^{\gamma_*(\varepsilon),\Upsilon_1, \Lambda}, \mathcal{U}_{p,r}^{\gamma_*(\varepsilon),\Upsilon_1}\right) \leq l_0 \mu(\Omega)\delta.
\end{eqnarray}

Denote
\begin{eqnarray}\label{eqq24}
\delta_* \left( \varepsilon\right) = \frac{\varepsilon}{10 l_0 \mu(\Omega)}.
\end{eqnarray}

From (\ref{eqq23}) and (\ref{eqq24}) it follows that for every partition $\Lambda$ such that $\delta \in (0,\delta_*(\varepsilon)]$ the inequality
\begin{eqnarray}\label{eqq25}
H_C\left( \mathcal{U}_{p,r}^{\gamma_*(\varepsilon),\Upsilon_1, \Lambda}, \mathcal{U}_{p,r}^{\gamma_*(\varepsilon),\Upsilon_1}\right) \leq \frac{\varepsilon}{10}
\end{eqnarray} is satisfied.

Note that the inequality (\ref{eqq25}) holds true for every finite $\Delta$-partition $\Upsilon_1$ of the compact set $\Omega.$ Finally, from (\ref{eqq16}) and (\ref{eqq25}) we obtain that for every finite $\Delta$-partition $\Upsilon_1$ of the compact set $\Omega$ and uniform
$\delta$-partition $\Lambda$ of the closed interval $\left[0,\gamma_*(\varepsilon)\right]$ such that  $\Delta \in (0, \Delta_*(\varepsilon)],$ $\delta \in (0, \delta_*(\varepsilon)],$ the inequality
\begin{eqnarray}\label{eqq26}
H_C\left( \mathcal{U}_{p,r}, \mathcal{U}_{p,r}^{\gamma_*(\varepsilon),\Upsilon_1,\Lambda}\right) \leq \frac{2\varepsilon}{5}
\end{eqnarray} is satisfied where $\Delta_*(\varepsilon)$ is defined by (\ref{eqq15}).

\vspace{3mm}

\textbf{Step 6.} Let us show that for given $\gamma_*(\varepsilon)>0,$ finite $\Delta$-partition $\Upsilon_1 =\left\{\Omega_1, \Omega_2, \ldots , \Omega_M\right\}$ of the set $\Omega$, uniform $\delta$-partition $\Lambda =\left\{0=w_0<w_1<\ldots <w_q =\gamma_*(\varepsilon)\right\}$ of the closed interval $\left[0,\gamma_*(\varepsilon)\right]$ where $\delta =w_{\lambda+1}-w_{\lambda},$ $\lambda= 0,1,\ldots, q-1$, and $\sigma >0$ the inclusion
\begin{eqnarray}\label{eqq27}
\mathcal{U}_{p,r}^{\gamma_*(\varepsilon),\Upsilon_1, \Lambda} \subset  \mathcal{U}_{p,r}^{\gamma_*(\varepsilon),\Upsilon_1,\Lambda,\sigma} + \gamma_*(\varepsilon)\sigma B_C(1)
\end{eqnarray} is satisfied where the set $\mathcal{U}_{p,r}^{\gamma_*(\varepsilon),\Upsilon_1,\Lambda,\sigma}$ is defined by (\ref{eqq2}).

Choose an arbitrary $\overline{y}(\cdot)\in \mathcal{U}_{p,r}^{\gamma_*(\varepsilon),\Upsilon_1, \Lambda}$ which is the image of $\overline{x}(\cdot) \in V_{p,r}^{\gamma_*(\varepsilon),\Upsilon_1, \Lambda}.$ From (\ref{v1}) we have that
\begin{eqnarray}\label{eqq28}
\overline{x}(s)=w_{\lambda_j}a_j, \ s \in \Omega_j,
\end{eqnarray}
where $w_{\lambda_j} \in \Lambda,$ $a_j \in S=\left\{x\in \mathbb{R}^m:\left\|x\right\|=1\right\},$  $j=1,2,\ldots, M,$ $\sum_{j=1}^{M} \mu(\Omega_j) w_{\lambda_j}^{p} \leq r^p.$ Since $S_{\sigma}$ is a finite $\sigma$-net on $S$, then for each $a_j \in S$ there exists $b_{\lambda_j}\in S_{\sigma}$ such that $\left\|a_j-b_{\lambda_j}\right\| \leq \sigma.$ Define new function $\overline{x}_*(\cdot):\Omega \rightarrow \mathbb{R}^m$ setting
\begin{eqnarray}\label{eqq29}
\overline{x}_*(s)=w_{\lambda_j}b_{\lambda_j}, \ s \in \Omega_j,
\end{eqnarray}
where $j=1,2,\ldots, M.$ From (\ref{eqq1}), (\ref{eqq28}) and (\ref{eqq29}) it follows that $\overline{x}_*(\cdot) \in$ \linebreak $ V_{p,r}^{\gamma_*(\varepsilon),\Upsilon_1, \Lambda,\sigma}$ and
\begin{eqnarray*}
\left\| \overline{x}(s)-\overline{x}_*(s)\right\| \leq w_{\lambda_j}\left\|a_j-b_{\lambda_j}\right\| \leq \gamma_*(\varepsilon)\sigma
\end{eqnarray*}  for every $s \in \Omega_j$ and $j=1,2,\ldots,M$ and hence
\begin{eqnarray}\label{eqq30}
\left\| \overline{x}(s)-\overline{x}_*(s)\right\| \leq \gamma_*(\varepsilon)\sigma
\end{eqnarray} for every $s\in \Omega.$

Let $\overline{y}_*(\cdot):E\rightarrow \mathbb{R}^n$ be the image of $\overline{x}_*(\cdot)\in V_{p,r}^{\gamma_*(\varepsilon),\Upsilon_1, \Lambda,\sigma}$ defined by (\ref{eqq29}). Then $\overline{y}_*(\cdot) \in \mathcal{U}_{p,r}^{\gamma_*(\varepsilon),\Upsilon_1, \Lambda,\sigma}$ and condition 2.B, (\ref{oper}) and (\ref{eqq30})
imply that
\begin{eqnarray*}
\left\| \overline{y}(\xi)-\overline{y}_*(\xi)\right\| \leq \int_{\Omega}l_0 \left\| \overline{x}(s)-\overline{x}_*(s)\right\|ds \leq l_0 \mu(\Omega)\gamma_*(\varepsilon)\sigma
\end{eqnarray*}  for every $\xi \in E$ and hence
\begin{eqnarray*}
\left\| \overline{y}(\cdot)-\overline{y}_*(\cdot)\right\|_C \leq l_0 \mu(\Omega)\gamma_*(\varepsilon)\sigma.
\end{eqnarray*}

Since $\overline{y}(\cdot) \in \mathcal{U}_{p,r}^{\gamma_*(\varepsilon),\Upsilon_1, \Lambda}$ is an arbitrarily chosen function, $\overline{y}_*(\cdot) \in \mathcal{U}_{p,r}^{\gamma_*(\varepsilon),\Upsilon_1, \Lambda,\sigma},$ then from the last inequality we obtain the validity of the inclusion (\ref{eqq27}). From (\ref{eqq27}) and inclusion $\mathcal{U}_{p,r}^{\gamma_*(\varepsilon),\Upsilon_1,\Lambda,\sigma} \subset  \mathcal{U}_{p,r}^{\gamma_*(\varepsilon),\Upsilon_1,\Lambda}$ it follows that
\begin{eqnarray}\label{eqq31}
H_C\left(\mathcal{U}_{p,r}^{\gamma_*(\varepsilon),\Upsilon_1, \Lambda},   \mathcal{U}_{p,r}^{\gamma_*(\varepsilon),\Upsilon_1,\Lambda,\sigma}\right) \leq l_0 \mu(\Omega)\gamma_*(\varepsilon)\sigma.
\end{eqnarray}

Let
\begin{eqnarray}\label{sigma*}
\sigma_*(\varepsilon)=\sigma_*\left(\varepsilon,\gamma_*\left(\varepsilon\right)\right)=\frac{\varepsilon}{10 l_0 \mu(\Omega)\gamma_*(\varepsilon)}
\end{eqnarray}

(\ref{eqq31}) and (\ref{sigma*}) yield that for every $\sigma \in \left(0,\sigma_*\left(\varepsilon \right)\right]$ the inequality
\begin{eqnarray}\label{eqq32}
H_C\left(\mathcal{U}_{p,r}^{\gamma_*(\varepsilon),\Upsilon_1, \Lambda},   \mathcal{U}_{p,r}^{\gamma_*(\varepsilon),\Upsilon_1,\Lambda,\sigma}\right) \leq  \frac{\varepsilon}{10}
\end{eqnarray} holds.

From (\ref{eqq26}) and (\ref{eqq32}) we conclude that for every finite $\Delta$-partition $\Upsilon_1$ of the compact set $\Omega$ and
uniform $\delta$-partition $\Lambda$ of the closed interval $\left[0,\gamma_*(\varepsilon)\right]$ such that  $\Delta \in (0, \Delta_*(\varepsilon)],$ $\delta \in (0, \delta_*(\varepsilon)]$ and for every $\sigma \in \left(0,\sigma_*(\varepsilon)\right]$ the inequality
\begin{eqnarray}\label{eqq33}
H_C\left( \mathcal{U}_{p,r}, \mathcal{U}_{p,r}^{\gamma_*(\varepsilon),\Upsilon_1,\Lambda,\sigma}\right) \leq \frac{\varepsilon}{2}
\end{eqnarray} is verified.

Thus, validity of the inequality (\ref{os1}) is proved.
\vspace{3mm}

\textbf{Step 7.} Now, in the last step the validity of the inequality (\ref{os2}) will be proved.

It is obvious that for given $\gamma_*(\varepsilon)>0$, for every finite
$\Delta$-partitions $\Upsilon_1 =\big\{\Omega_1, \Omega_2, \ldots , \Omega_M\big\}$ of the compact set $\Omega\subset \mathbb{R}^k$ and  $\Upsilon_2 =\left\{E_1, E_2, \ldots, E_N\right\}$ of the compact set $E\subset \mathbb{R}^b,$ uniform $\delta$-partition $\Lambda =\big\{0=w_0<w_1<\ldots <w_q =\gamma_*(\varepsilon)\big\}$ of the closed interval $\left[0,\gamma_*(\varepsilon)\right]$ where $\delta =w_{\lambda+1}-w_{\lambda},$ $\lambda= 0,1,\ldots, q-1$, and $\sigma >0$ the inclusion
\begin{eqnarray}\label{eq66}
\mathcal{F}_{p,r}^{\gamma_*(\varepsilon), \Upsilon_1, \Lambda, \sigma, \Upsilon_2} \subset \mathcal{F}_{p,r}
\end{eqnarray} is satisfied where $\gamma_*(\varepsilon)>0$ is defined by (\ref{eq60}), $\mathcal{F}_{p,r}$ is the integral funnel of the closed ball $V_{p,r}\subset L_p(\Omega;\mathbb{R}^m)$ under operator (\ref{oper}) and is defined by (\ref{eq3}), the set $\mathcal{F}_{p,r}^{\gamma_*(\varepsilon), \Upsilon_1, \Lambda, \sigma, \Upsilon_2}$ consists of a finite number of points and is defined by (\ref{eqq4}).

Now let $\gamma_*(\varepsilon)>0$ be defined by (\ref{eq60}), finite $\Delta$-partitions $\Upsilon_1 =\left\{\Omega_1, \Omega_2, \ldots , \Omega_M\right\}$ of the compact set $\Omega$ and  $\Upsilon_2 =\left\{E_1, E_2, \ldots, E_N\right\}$ of the compact set $E\subset \mathbb{R}^b$ be such that $\Delta \in \left(0,\Delta_*(\varepsilon)\right]$, uniform $\delta$-partition $\Lambda =\left\{0=w_0<w_1<\ldots <w_q =\gamma_*(\varepsilon)\right\}$ of the closed interval $\left[0,\gamma_*(\varepsilon)\right]$  be such that $\delta \in \left(0,\delta_*(\varepsilon)\right],$ and $\sigma \in \left(0, \sigma_*(\varepsilon)\right]$ where $\delta =w_{\lambda+1}-w_{\lambda},$ $\lambda= 0,1,\ldots, q-1,$ $\Delta_*(\varepsilon)>0,$ $\delta_*(\varepsilon)>0$ and $\sigma_*(\varepsilon)>0$ are defined by (\ref{eqq15}), (\ref{eqq24}), (\ref{sigma*}) respectively. By virtue of (\ref{eqq33}) we have
\begin{eqnarray}\label{eqq33*}
H_n\left( \mathcal{U}_{p,r}(\xi), \mathcal{U}_{p,r}^{\gamma_*(\varepsilon),\Upsilon_1,\Lambda,\sigma}(\xi)\right) \leq \frac{\varepsilon}{2}
\end{eqnarray} is held for every $\xi \in E$ where the sets $\mathcal{U}_{p,r}(\xi)$ and  $\mathcal{U}_{p,r}^{\gamma_*(\varepsilon),\Upsilon_1,\Lambda,\sigma}(\xi)$ are defined by (\ref{eq2}) and (\ref{eqq3}) respectively.

Now, let us prove that
\begin{eqnarray}\label{eqq67}
\mathcal{F}_{p,r} \subset \mathcal{F}_{p,r}^{\gamma_*(\varepsilon), \Upsilon_1, \Lambda, \sigma, \Upsilon_2} + \frac{5\varepsilon}{6} \overline{B}_{k+n}(1)
\end{eqnarray}
where $\overline{B}_{k+n}(1)= \left\{z\in \mathbb{R}^{k+n}:\left\| z \right\| \leq 1\right\}.$

Choose an arbitrary $(\xi_*, z_*) \in \mathcal{F}_{p,r}.$ Then we have that $z_* \in \mathcal{U}_{p,r}(\xi_*).$ By virtue of (\ref{eqq33*}) we have that there exists $w_*\in \mathcal{U}_{p,r}^{\gamma_*(\varepsilon),\Upsilon_1, \Lambda, \sigma}(\xi_*)$ such that
\begin{eqnarray}\label{eq663}
\left\|z_*-w_*\right\| < \frac{3\varepsilon}{5}.
\end{eqnarray}

Since $\Upsilon_2 =\left\{E_1, E_2, \ldots, E_N\right\}$ is a finite $\Delta$ partition of $E,$ then by virtue of the Definition \ref{def11}  we have that there exists $i_*$ such that  $\xi_* \in E_{i_*}$ and $\left\|\xi_*-\xi_{i_*}\right\| \leq \Delta$. Similarly to the Proposition \ref{prop11} it is possible to show that
\begin{eqnarray}\label{eq665}
H_n \left(\mathcal{U}_{p,r}^{\gamma_*(\varepsilon), \Upsilon_1, \Lambda, \sigma}(\xi_*),\mathcal{U}_{p,r}^{\gamma_*(\varepsilon), \Upsilon_1, \Lambda, \sigma}(\xi_{i_*})\right) \leq \beta_* \cdot \varphi \left(\left\|\xi_*-\xi_{i_*}\right\|\right).
\end{eqnarray}

Since $\Delta \in \left(0,\Delta_*(\varepsilon)\right),$ then from (\ref{eqq15*}), (\ref{eqq15}) and (\ref{eq665}) we obtain that
for $w_*\in \mathcal{U}_{p,r}^{\gamma_*(\varepsilon),\Upsilon_1, \Lambda, \sigma}(\xi_*)$ there exists $f_*\in \mathcal{U}_{p,r}^{\gamma_*(\varepsilon),\Upsilon_1, \Lambda, \sigma}(\xi_{i_*})$ such that
\begin{eqnarray}\label{eq667}
\left\| w_*-f_*\right\| < \frac{\varepsilon}{9}.
\end{eqnarray}

It is obvious that $(\xi_{i_*},f_*)\in \mathcal{F}_{p,r}^{\gamma_*(\varepsilon),\Upsilon_1, \Lambda, \sigma, \Upsilon_2}.$
Since $\left\|\xi_*-\xi_{i_*}\right\| \leq \Delta <\Delta_*(\varepsilon)$, then (\ref{eqq15}), (\ref{eq663})  and (\ref{eq667}) imply
\begin{eqnarray}\label{eq668}
\left\|(\xi_*,z_*)-(\xi_{i_*},f_*)\right\| &\leq & \left\|\xi_*-\xi_{i_*}\right\| +\left\|z_*-f_*\right\|\nonumber \\ &\leq & \left\|\xi_*-\xi_{i_*}\right\| +\left\|z_*-w_*\right\|+\left\|w_*-f_*\right\|<\frac{5\varepsilon}{6}.
\end{eqnarray}
Thus we have proved that for an arbitrary chosen $(\xi_*, z_*) \in \mathcal{F}_{p,r}$ there exists $(\xi_{i_*},f_*)\in \mathcal{F}_{p,r}^{\gamma_*(\varepsilon),\Upsilon_1, \Lambda, \sigma, \Upsilon_2}$ such that the inequality (\ref{eq668}) is held. This  means that the inclusion (\ref{eqq67}) is verified.

Finally, from (\ref{eq66}) and (\ref{eqq67}) we have the validity of the inequality (\ref{os2}). The proof of the theorem is completed.
\end{proof}

From Theorem \ref{teo41} it follows that for each $\Delta \in \left(0,\Delta_*(\varepsilon)\right],$ $\delta \in \left(0,\delta_*(\varepsilon)\right]$ and $\sigma \in (0,\sigma_*(\varepsilon)]$ the inequality
\begin{eqnarray*}
H_{n}\left(\mathcal{U}_{p,r}(\xi), \mathcal{U}_{p,r}^{\gamma_*(\varepsilon),\Upsilon_1, \Lambda,\sigma}(\xi)\right) < \varepsilon
\end{eqnarray*}
is also satisfied for every $\xi \in E$ where the set $\mathcal{U}_{p,r}(\xi)$ is defined by (\ref{eq2}), the set $\mathcal{U}_{p,r}^{\gamma_*(\varepsilon), \Upsilon_1, \Lambda,\sigma}(\xi)$ is defined by (\ref{eqq3}), $\gamma_*(\varepsilon)$ is defined by (\ref{eq60}), $\Delta_*(\varepsilon)$ is defined by (\ref{eqq15}), $\delta_*(\varepsilon)$ is defined by (\ref{eqq24}), $\sigma_*(\varepsilon)$ is defined by (\ref{sigma*}),  $\Upsilon_1$ is a finite $\Delta$-partition of the compact set $\Omega,$ $\Lambda$ is a uniform partition of the closed interval  $[0,\gamma_*(\varepsilon)],$ $\delta$ is its diameter.

Note that since
\begin{eqnarray*}
\mathcal{U}_{p,r}^{\gamma_*(\varepsilon),\Upsilon_1, \Lambda,\sigma} \subset \mathcal{U}_{p,r} \, ,
\ \ \mathcal{F}_{p,r}^{\gamma_*(\varepsilon),\Upsilon_1, \Lambda,\sigma, \Upsilon_2} \subset \mathcal{F}_{p,r}
\end{eqnarray*} we conclude that the presented approximations are internal ones.


\begin{thebibliography}{}

\bibitem{bla}
Blagodatskikh, V.I.,  Filippov, A.F.: Differential inclusions and optimal control. Topology, ordinary differential equations, dynamical systems. Trudy Mat. Inst. Steklov. 169, 194-252, 1985.

\bibitem{bel}
Beletskii, V.V.: Notes on the Motion of Celestial Bodies. Nauka, Moscow, 1972.


\bibitem{gus1}
Kh.G.Guseinov, Kh.G., Nazlipinar A.S.: An algorithm for approximate calculation of the attainable sets of the nonlinear control systems with integral constraint on controls.  Comput. Math. Appl. 62(1), 1887-1895, 2011.

\bibitem{gusev}
Gusev, M.I., Zykov, I.V.: On extremal properties of the boundary points of reachable sets for control systems
with integral constraints. Tr. Inst. Math. Mekh. UrO RAN 23(1), 103-115, 2017.

\bibitem{hus1}
Huseyin, N., Guseinov, Kh.G., Ushakov, V.N.: Approximate construction of the set of trajectories of the control system described by a Volterra integral equation. Math. Nachr. 288(16), 1891-1899, 2015.


\bibitem{hus2}
Huseyin, A., Huseyin, N., Guseinov, Kh.G..: Approximation of sections of the set of trajectories for a control system with bounded control resources. Trudy Inst. Mat. Mekh. UrO RAN, 23(1),  116-127, 2017.

\bibitem{hus3}
Huseyin, A.: Approximation of the integral funnel of the Urysohn type integral operator. Appl. Math. Comput. 341, 277-287, 2019.

\bibitem{hus4}
Huseyin, A., Huseyin, N., Guseinov Kh.G.: Approximation of the integral funnel of a nonlinear control system with limited control resources. Minimax Theory Appl. 5(2), 327-346, 2020.

\bibitem{hus5}
Huseyin, N., Huseyin, A.,  Guseinov, Kh.G.: Approxmation of the set of trajectories of the nonlinear control system with limited control resources. Math. Model. Anal. 23(1), 152-166, 2018.

\bibitem{kan}
Kantorovich, L.V., Akilov, G.P.: Functional Analysis. Nauka, Moscow, 1977.

\bibitem{kra}
Krasovskii, N.N.: Theory of Control of Motion: Linear Systems. Nauka, Moscow, 1968.

\bibitem{pan}
Panasyuk, A.I., Panasyuk, V.I.: An equation generated by a differential inclusion.  Mat. Zametki 27(3),  429-437, 1980.

\bibitem{sub}
Subbotina, N.N.,  Subbotin, A.I.: Alternative for the encounter-evasion differential game with constraints on the momenta of the players controls. J. Appl. Math. Mech. 39(3), 376-385, 1975.


\bibitem{whe}
Wheeden, R.L., Zygmund, A.: Measure and Integral. An Introduction to Real Analysis. M. Dekker Inc., New York, 1977.

\end{thebibliography}
\end{document}